\documentclass[11pt,a4paper,leqno]{amsart}

\usepackage[latin1]{inputenc}
\usepackage[T1]{fontenc}
\usepackage{amsfonts}
\usepackage{amsmath}
\usepackage{amssymb}
\usepackage{eurosym}
\usepackage{mathrsfs}
\usepackage{color}
\usepackage{esint}
\usepackage{url}

\usepackage{enumerate}

\usepackage[pagebackref,hypertexnames=false, colorlinks, citecolor=blue, linkcolor=blue, urlcolor=red]{hyperref}

\newcommand{\R}{\mathbb{R}}

\newcommand{\Q}{\mathbb{Q}}
\newcommand{\N}{\mathbb{N}}

\newcommand{\calA}{\mathcal{A}}
\newcommand{\calF}{\mathcal{F}}
\newcommand{\calK}{\mathcal{K}}
\newcommand{\calM}{\mathcal{M}}
\newcommand{\calR}{\mathcal{R}}

\numberwithin{equation}{section}

\newcommand{\ud}[0]{\,\mathrm{d}}

\newcommand{\esssup}[0]{\operatornamewithlimits{ess\,sup}}

\newcommand{\abs}[1]{|#1|}

\newcommand{\OpN}[1]{| #1 |_{\operatorname{op}}}

\newcommand{\conv}[0]{\operatorname{conv}}

\newcommand{\loc}[0]{\operatorname{loc}}

\newcommand{\sign}[0]{\operatorname{sgn}}

\theoremstyle{plain}
\newtheorem{thm}[equation]{Theorem}
\newtheorem{lem}[equation]{Lemma}

\theoremstyle{definition}

\theoremstyle{remark}
\newtheorem{rem}[equation]{Remark}

\allowdisplaybreaks[1]

\title{The Strong Matrix Weighted Maximal Operator}

\author{Emil Vuorinen}

\address{Department of Mathematics and Statistics, University of Helsinki, P.O.B. 68, FI-00014 University of Helsinki, Finland}
\email{emil.vuorinen@helsinki.fi}

\makeatletter
\@namedef{subjclassname@2020}{%
  \textup{2020} Mathematics Subject Classification}
\makeatother


\thanks{The author was supported by the Academy of Finland via the Finnish Centre of Excellence in Randomness and Stuctures (project 346314).}

\begin{document}

\begin{abstract}
The purpose of this note is to prove that the strong Christ-Goldberg maximal function is bounded.
This is a matrix weighted maximal operator appearing in the theory of matrix weighted norm inequalities.
Related to this we record the Rubio de Francia extrapolation theorem with
bi-parameter matrix weights.
\end{abstract}

\maketitle

\tableofcontents

\section{Introduction}

This note is related to \emph{matrix weighted norm inequalities}. 
Let $W \colon \R^n \to \calM_d$ be a measurable mapping, where $\calM_d$ is the collection of $d \times d$ matrices
with real entries. We assume that $W(x)$ is symmetric and positive definite for almost every $x$. Such
mappings are called \emph{matrix weights}. The theory of matrix weighted norm inequalities studies
for example estimates of the form
\begin{equation}\label{eq17}
\int_{\R^n} | W(x)^{1/p} Tf(x)|^p \ud x
\lesssim  \int_{\R^n} |W(x)^{1/p} f(x) |^p\ud x.
\end{equation}
Here $p \in (1, \infty)$, $f = (f_1, \dots, f_d) \colon \R^n \to \R^d$, 
$T$ is some linear operator of Harmonic Analysis such as a singular integral, 
$Tf:= (Tf_1, \dots, Tf_d)$ and $|\cdot|$ denotes the Euclidean norm on $\R^d$. 
We refer for example to the introduction of the recent article Bownik, Cruz-Uribe \cite{BC} and the 
references therein for information on this topic and its history.

We break with the usual convention of writing matrix weighted estimates as in \eqref{eq17}.
Instead, we adopt the formulation
\begin{equation}\label{eq18}
\int_{\R^n} | W(x) Tf(x)|^p \ud x
\lesssim  \int_{\R^n} |W(x) f(x) |^p\ud x,
\end{equation}
which was used also in \cite{BC}; see the Introduction of \cite{BC}. 
One reason 
to use this formulation is that it fits well with \emph{extrapolation}, which we 
discuss more below.
Since we are going to state a bi-parameter extrapolation theorem
we will write matrix weighted estimates in the form \eqref{eq18}. 

A certain matrix weighted maximal operator was introduced in Christ, Goldberg \cite{CG} and
Goldberg \cite{Goldberg}. Let $W$ be a matrix weight. The \emph{Christ-Goldberg maximal operator}
is defined by
\begin{equation}\label{eq:DefM_W}
M_W f(x):= \sup_{Q \owns x} \fint_Q | W(x) W(y)^{-1} f(y) | \ud y.
\end{equation}
Here again $f \colon \R^n \to \R^d$, the supremum is over cubes $Q \subset \R^n$, $Q \owns x$, with
sides parallel to the axes, $\fint_Q = \frac{1}{m(Q)} \int_Q$ is the average over $Q$ and $m$ is the
Lebesgue measure. We point out that using \eqref{eq18} instead of \eqref{eq17}
affects the definition of $M_W$. (The definition of $M_W$ corresponding to the 
formulation \eqref{eq17} is obtained by replacing $W$ by $W^{1/p}$ in \eqref{eq:DefM_W}, see for example \cite{Goldberg}.)
Notice that $M_Wf(x) \in [0, \infty]$. It was shown in \cite{CG,Goldberg} that the estimate
\begin{equation}\label{eq13}
\int_{\R^n} | M_W f(x)|^p \ud x
\lesssim \int_{\R^n} | f (x) | ^p \ud x
\end{equation}
holds for $p \in (1, \infty)$,
if $W$ is a \emph{matrix $\calA_p$ weight} (see Section \ref{sec1} for  definitions).

In this note we consider the \emph{strong Christ-Goldberg maximal function} $M_W^s$,
also called the \emph{bi-parameter Christ-Goldberg maximal function}. 
Let $\R^n = \R^{n_1} \times \R^{n_2}$. The collection of \emph{bi-parameter rectangles} $\calR$
consists of the rectangles $R=Q_1 \times Q_2\subset \R^n$, where $Q_i \subset \R^{n_i}$ is a cube.
The maximal operator $M_W^s$ is defined as the one-parameter version $M_W$ in \eqref{eq:DefM_W}
except that one takes the supremum over rectangles $R \in \calR$, $R \owns x$,  instead of cubes $Q$.
The main purpose of this note is to prove the following theorem.

\begin{thm}\label{thm:StrongCG}
Let $p \in (1, \infty]$ and $W \in \calA_p(\R^{n_1} \times \R^{n_2})$. Then 
$$
\| M_W^s f \|_{L^p(\R^n)} 
\le C(n,d,p)[W]_{\calA_p(\R^{n_1} \times \R^{n_2})}^{2p'} \| f \|_{L^p(\R^n;\R^d)}.
$$
If $p=\infty$, the estimate holds with $[W]_{\calA_\infty(\R^{n_1} \times \R^{n_2})}$ instead of 
$[W]_{\calA_\infty(\R^{n_1} \times \R^{n_2})}^2$.
\end{thm}

We point out that it is not known if the constant obtained in Theorem \ref{thm:StrongCG} is sharp 
in terms of the exponent of $[W]_{\calA_p(\R^{n_1} \times \R^{n_2})}$; see Remark \ref{rem:sharpness}.

In the scalar case $d=1$ it is well known that estimates for bi-parameter maximal functions
follow from the corresponding one-parameter estimates via iteration. Indeed, let
\begin{equation}\label{M^s}
M^s f(x)=\sup_{R \owns x} \fint_R | f(y) | \ud y
\end{equation}
be the usual bi-parameter maximal function, 
where $f \colon \R^n \to \R$. Then
one sees that
\begin{equation}\label{eq14}
M^s f(x) \le M^1 M^2 f(x),
\end{equation}
where $M^i$ is the one-parameter maximal function acting on the parameter $i$. For example,
$$
M^2f(x_1,x_2):= \sup_{\substack{Q \subset \R^{n_2} \\ Q \owns x_2}} \fint_{Q} |f(x_1, y_2)| \ud y_2.
$$
Weighted estimates 
$$
\int_{\R^n} |w(x) M^s f (x)|^p \ud x
\lesssim \int_{\R^n} |w(x) f(x)|^p \ud x,
$$
where $p \in (1, \infty)$ and $w$ is a bi-parameter $\calA_p$ weight, 
follow from the domination \eqref{eq14} and the corresponding one-parameter estimates.
This uses the well known fact that if $w$ is a bi-parameter $\calA_p$ weight, then
$w(\cdot, x_2)$ and $w(x_1, \cdot)$ are one-parameter $\calA_p$ weights, see Lemma \ref{lem:UniformAp}.

However, in the matrix weighted case $d \ge 2$ iteration such as \eqref{eq14} does not work in the same way, and so
estimates for $M_W^s$ do not directly follow from the estimates of $M_W$. 
In \cite{BC} the \emph{Rubio de Francia extrapolation theorem} was proved for matrix weights. 
Exrapolation is a method which implies weighted estimates in a range of exponents from
weighted estimates with a single exponent.
This important result had been long missing in the matrix weighted theory. 
One of the key ingredients in \cite{BC} is to use a certain \emph{convex set-valued maximal operator $M_\calK$}.
The operator $M_\calK$ acts on set-valued functions $F$ giving set-valued functions $M_\calK F$.
This allows iteration: one can form the functions $M_\calK F$, $M_\calK M_\calK F=M_\calK^2 F$ and so on,
which is important for the proof of the extrapolation.

Since the convex set-valued maximal operator works well with iteration, 
we can prove matrix weighted estimates for the bi-parameter version $M_\calK^s$ of the operator $M_\calK$.
Indeed, if $F$ is a suitable set-valued function,
then
$$
M_\calK^s F(x) \subset M_\calK^1 M_\calK^2 F (x),
$$
where $M_\calK^i$ is the one-parameter operator acting in the parameter $i$.
This corresponds to \eqref{eq14} and is proved in Lemma \ref{lem:BiParDomination}. From here one 
obtains matrix weighted estimates for $M^s_\calK$
from the corresponding estimates of $M_\calK$ in a similar manner as in the case $d=1$ described above.
This is done in Theorem \ref{thm:StrongConvex}.

The convex  set-valued maximal operators and the Christ-Goldberg maximal functions
are equivalent in a certain pointwise manner, see Lemma \ref{lem1} and Equation \eqref{eq4}. 
Therefore, the matrix weighted estimates we obtain
for $M_\calK^s$ imply Theorem \ref{thm:StrongCG}, which is proved in Section \ref{sec:MainThms}.
In Theorem \ref{thm1} we record the fact that given a bi-parameter matrix  $\calA_p$ weight $W$,
there exists a $\delta > 0$ such that $M_W^s$ is bounded in $L^q$ for those $q \in (1, \infty)$
such that $|q-p| < \delta$; the one-parameter version of this was proved in \cite{Goldberg}.

Once one has the matrix weighted estimates for $M_\calK^s$, 
the extrapolation with bi-parameter matrix weights can be proved precisely as the 
one-parameter case in \cite{BC}. This is stated in Theorem \ref{thm:BiParExtrap}.
As a corollary of this result, we record in Theorem \ref{thm:BiParJourne} the fact that \emph{bi-parameter Journ\'e operators}
satisfy matrix weighted estimates in the full range $p \in (1, \infty)$. This is a result which
had been proved in the case $p=2$ in Domelevo, Kakaroumpas, Petermichl, Soler i Gibert \cite[Theorem 1.1]{DKPS}
(in \cite{DKPS} the result is proved for general $p$ if $T$ is paraproduct free);
see the discussion in Section \ref{sec:bi-parExtra}.

\subsection*{Acknowledgements}
The author thanks Timo H\"anninen and Tuomas Hyt\"onen for discussions around this topic.
The author also thanks the referee for valuable suggestions which improved the paper.

\section{Definitions and Background}\label{sec1}

We begin by setting up notation and collecting background results. 

\subsection{Matrix weights}\label{sec:MatrixWeights}
For more information on matrix weighted estimates and their history we refer for example to the recent paper
\cite{BC} and the references therein.  Here we only recall the definitions needed for this note.

Let $\calM_d$ denote the collection of all $d \times d$ matrices with  real entries. 
A matrix weight $W$ is a measurable function $W \colon \R^n \to \calM_d$
such that $W(x)$ is positive definite for almost every $x$. Let $p \in (1, \infty)$.  
A matrix weight $W$ is in the class $\calA_p$ if
\begin{equation}\label{eq15}
[W]_{\calA_p}
:= \sup_{Q} \Big( \fint_Q \Big( \fint_Q \OpN{W(x)W(y)^{-1}} ^{p'} \ud y \Big)^{p/p'} \ud x \Big)^{1/p} < \infty,
\end{equation}
where the supremum is over cubes in $\R^n$ with sides parallel to the axes. The norm $\OpN{\cdot}$
denotes the operator norm of a matrix.
The weight $W$ is in $\calA_\infty$ if 
$$
[W]_{\calA_\infty}
:= \sup_Q \esssup_{x \in Q} \fint_Q \OpN{W(x)W(y)^{-1}} \ud y < \infty,
$$
and it is in $\calA_1$ if
$$
[W]_{\calA_1} := \sup_Q \esssup_{y \in Q} \fint_Q \OpN{W(x) W(y)^{-1}} \ud x < \infty.
$$

The above definition in the case $p \in (1, \infty)$
is due to Roudenko \cite{Roudenko} and the case $p=1$ is from Frazier, Roudenko \cite{FR}.
The class $\calA_\infty$ is introduced in \cite{BC}.
There are different equivalent ways to characterise matrix $\calA_p$  weights, 
see for example \cite[Proposition 6.5]{BC}.

The usual way to state the condition for $p \in (1, \infty)$,
which is related to writing matrix weighted estimates as in \eqref{eq17},
is to say that a
matrix weight $W$ is in $A_p$ if
$$
[W]_{A_p}
:= \sup_{Q}  \fint_Q \Big( \fint_Q \OpN{W(x)^{1/p}W(y)^{-1/p}} ^{p'} \ud y \Big)^{p/p'} \ud x < \infty.
$$
So we see that if $p \in (1, \infty)$, $W \in \calA_p$ if and only if $W^p \in A_p$ and 
$[W]_{\calA_p}=[W^p]_{A_p}^{1/p}$. 

\subsection{The Christ-Goldberg maximal function}\label{sec:CGMax}
The Christ-Goldberg maximal function $M_W$ 
was defined in \eqref{eq:DefM_W}.
Let $p \in (1, \infty)$ and $W \in \calA_p$ be a matrix weight. 
In \cite[Theorem 3.2]{Goldberg} it was proved that there exists a $\delta > 0$ such that
\begin{equation}\label{eq16}
\| M_{W}f\|_{L^q(\R^n)} \le C(q) \| f\|_{L^q(\R^n;\R^d)}
\end{equation}
for all $q \in (1, \infty)$ such that $|q-p| < \delta$. 
The proof of \eqref{eq16} in \cite{Goldberg} is based on reverse H\"older 
inequalities of scalar ($d=1$) weights. 

Following the proof in \cite{Goldberg} using a quantitative version
of the reverse H\"older inequality, see for example Grafakos \cite[Theorem 7.2.2, Remark 7.2.3]{Grafakos},
the estimate for $M_W$ can be rewritten as follows.
Let $p \in (1, \infty)$ and $B > 1$. 
There exists a $\delta =  \delta(n,d,p,B)> 0$ so that
if $W \in \calA_p$ is a matrix weight such that $[W]_{\calA_p} \le B$, then
\begin{equation}\label{eq9}
\| M_{W}f\|_{L^q(\R^n)} \le C(n,d,q,B) \| f\|_{L^q(\R^n;\R^d)}
\end{equation}
for all $q$ such that $|q-p| < \delta$.
This more precise formulation will have relevance in Theorem \ref{thm1}.

The quantitative estimate
\begin{equation}\label{eq6}
\| M_W f \|_{L^p(\R^n)} \le C(n,d,p)[W]_{\calA_p}^{p'} \| f \|_{L^p(\R^n; \R^d)}
\end{equation}
was proved in Isralowitz, Moen \cite{IM} for $p \in (1, \infty)$ and in \cite{BC} for $p=\infty$.

\subsection{Set-valued functions}
Let $\calK(\R^d)$ denote the collection of closed subsets of $\R^d$. 
We will consider measurable functions $F \colon \R^n \to \calK(\R^d)$. 
For the definition of measurability in this context see for example \cite[Definition 3.1]{BC}. 
Different ways to characterise
the measurability of such functions are stated in \cite[Theorem 3.2]{BC}, whose proof can be found in 
Aubin, Frankowska \cite[Theorems 8.1.4, 8.3.1]{AF}. Here we use the following: 
A function $F \colon \R^n \to \calK(\R^d)$ is measurable if and only if there exist measurable
functions $f_k \colon \R^n \to \R^d$, $k \in \N$, such that
$$
F(x)=\overline{\{f_k(x)\colon k \in \N\}}, \quad \text{for every } x \in \R^n.
$$

If $A \subset \R^d$, we define
$$
|A| = \sup\{| a | \colon a \in A\}.
$$
This is not to be confused with the Lebesgue measure of a set.
A measurable function $F \colon \R^d \to \calK(\R^d)$ is \emph{integrably bounded}
if there exists a non-negative function $k \in L^1(\R^n)$ so that $|F(x)| \le k(x)$ for every $x$; 
we say that $F$ is \emph{locally integrably bounded} if $k \in L^1_{\loc}(\R^n)$.

Next we define the \emph{Aumann integral} of set-valued functions.
Let $F \colon \R^n \to \calK(\R^d)$ be measurable. We denote by $S^0(F)$ the collection of 
\emph{measurable selection functions of $F$}. These are the measurable vector-valued functions
$f \colon \R^n \to \R^d$ such that $f(x) \in F(x)$ for every $x$. The set of \emph{integrable selection functions},
that is, those functions $f \in S^0(F)$ such that $\int | f | < \infty$, 
is denoted by $S^1(F)$. The Aumann integral of $F$ is defined by
$$
\int F \ud x
=\Big\{ \int f \ud x \colon f \in S^1(F)\Big\}.
$$
In general, $F$ may not have integrable selection functions. If $F$ is integrably bounded, then
$S^0(F)=S^1(F)$. For further properties of the Aumann integral we refer to \cite{AF} and \cite{BC}.

We will denote the collection of bounded, convex and symmetric sets in $\calK(\R^d)$,
that is, compact, convex and symmetric sets of $\R^d$,  by $\calK_{bcs}(\R^d)$. 
Let $W \colon \R^n \to \calM_d$ be
a matrix weight.
For $p \in (0, \infty]$ the \emph{$L^p$ space $L^p_\calK(\R^n,W)$ of convex set-valued functions} (see \cite[Definition 4.7]{BC})
is the collection of
those measurable $F \colon \R^n \to \calK_{bcs}(\R^d)$ such that
$$
\| F \|_{L^p_\calK(\R^n,W)} :=  \| |W F | \|_{L^p(\R^n)} < \infty,
$$
where the last quantity is the usual $L^p$ norm of the scalar function $|WF|$
and
$$
|W(x)F(x)|
=\sup\{ | a | \colon a \in W(x) F(x)\}
=\sup \{|W(x)b| \colon b \in F(x)\}.
$$
We refer to \cite{BC} for further properties of $L^p_\calK(\R^n,W)$. If 
$W(x) = I(x)$ (the identity matrix) for every $x$, then we denote 
$L^p_\calK(\R^n,W)=L^p_\calK(\R^n)$.

\subsection{The convex set-valued maximal operator}\label{sec:CSMO}
If $E \subset \R^n$ is a set with positive and finite measure, 
and $f \colon \R^n \to \R^d$ is integrable over $E$,
we will denote the average of $f$ over $E$ by
$$
\langle f \rangle_E = \fint_E f \ud x
=\frac{1}{m(E)} \int_E f \ud x.
$$
Let $F \colon \R^n \to \calK_{bcs}(\R^d)$ be locally integrably bounded.
If $Q \subset \R^n$ is a cube, then
$$
\langle F \rangle_Q
:= \frac{1}{m(Q)} \int 1_Q F \ud x
= \{ \langle f \rangle_Q \colon f \in S^0(F)\},
$$
where the last equality follows from the definition of the Aumann integral.
Notice that since $F$ is locally integrably bounded, every $ f \in S^0(F)$ is locally integrable, and therefore
we can use $S^0(F)$ in the above formula instead of $S^1(1_QF)$.
The convex set-valued maximal operator 
is defined
by
$$
M_\calK F(x)
:= \overline{\conv}\Big(\bigcup_{Q \subset \R^n} \langle F \rangle_Q 1_Q(x) \Big)
=\overline{\conv}(\{ \langle f \rangle_Q 1_Q(x) \colon Q \subset \R^n, f \in S^0(F)\} ),
$$
where $Q$ ranges over cubes $Q \subset \R^n$ and $\overline{\conv}$ denotes 
the closed convex hull.
See \cite[Definition 5.4]{BC}.

By \cite[Proposition 5.7]{BC} $M_\calK F \colon \R^n \to \calK_{cs}(\R^d)$ is measurable,
where $\calK_{cs}(\R^d)$ consists of the closed, convex and symmetric sets. By \cite[Lemma 5.5]{BC}
$M_\calK$ is sublinear in the sense that if
$F,G \colon \R^n \to \calK_{bcs}(\R^d)$ are locally integrably bounded and $\alpha \in \R$,
then there holds that $M_\calK(F+G) \subset M_\calK F + M_\calK G$ and $M_\calK (\alpha F ) = \alpha M_\calK F$. Also, $M_\calK$ is
monotone: if $F \subset G$, then $M_\calK F \subset M_\calK G$.

In \cite[Theorem 6.9]{BC} it was shown that if $p \in (1, \infty]$ and $W \in \calA_p$, then
$M_\calK \colon L^p_{\calK}(\R^n,W) \to L^p_{\calK}(\R^n,W)$ and
\begin{equation}\label{eq:ConvexMBounded}
\| M_\calK F \|_{L^p_{\calK}(\R^n,W)} 
\le C(n,d,p) [W]_{\calA_p}^{p'} \| F \|_{L^p_{\calK}(\R^n,W)}. 
\end{equation}

\subsection{Bi-parameter matrix weights}

Let $\R^n=\R^{n_1} \times \R^{n_2}$, where $n_1, n_2 \in \{1,2,3\dots\}$. We often write 
a point $x \in \R^n$ as $x=(x_1,x_2)$, where $x_i \in \R^{n_i}$. 
Recall that $\calR$ is the collection of bi-parameter rectangles $R=Q_1\times Q_2$, 
where $Q_i \subset \R^{n_i}$ is a cube.

For $p \in [1, \infty]$ the class $\calA_p(\R^{n_1} \times \R^{n_2})$ 
of \emph{bi-parameter matrix weights} is defined like the one-parameter matrix $\calA_p$ weights 
in Section \ref{sec:MatrixWeights} except for replacing cubes with rectangles. For example,
if $p \in (1, \infty)$, a matrix weight $W \colon \R^{n_1} \times \R^{n_2} \to \calM_d$ 
is in $\calA_p(\R^{n_1} \times \R^{n_2})$ if
\begin{equation}\label{eq:BiParAp}
[W]_{\calA_p(\R^{n_1} \times \R^{n_2})}
:= \sup_{R\in \calR} \Big( \fint_R \Big( \fint_R \OpN{W(x)W(y)^{-1}} ^{p'} \ud y \Big)^{p/p'} \ud x \Big)^{1/p} < \infty.
\end{equation}
We refer to \cite{DKPS} for some basic
properties of bi-parameter matrix $\calA_p$ weights (in the case $p \in (1, \infty)$).

The following lemma is proved in \cite[Lemma 3.6]{DKPS}.

\begin{lem}\label{lem:UniformAp}
Let $p \in (1, \infty)$ and let $W \in \calA_p(\R^{n_1} \times \R^{n_2})$ be a bi-parameter matrix weight.
Then, for almost every $x_1 \in \R^{n_1}$, the function $W(x_1, \cdot)$ is a one-parameter matrix weight with
$$
[W(x_1, \cdot)]_{\calA_p(\R^{n_2})} \lesssim_{d,n,p} [W]_{\calA_p(\R^{n_1} \times \R^{n_2})}.
$$
The symmetric statement holds for $W(\cdot, x_2)$ for almost every $x_2 \in \R^{n_2}$.
\end{lem}

\subsection{Bi-parameter maximal functions}
Bi-parameter versions of the maximal operators are also obtained by replacing cubes with rectangles. 
The bi-parameter Christ-Goldberg maximal function 
or the strong Christ-Goldberg maximal function
is defined by 
\begin{equation}\label{eq:BiParCGMax}
M_{W}^s f(x) := \sup_{R \owns x} 
\fint_R \abs{W(x)W(y)^{-1} f(y)} \ud y,
\end{equation}
where $R \in \calR$,  $f \colon \R^{n_1} \times \R^{n_2} \to \R^d$ and 
$W \colon \R^{n_1} \times \R^{n_2} \to \calM_d$ is a matrix weight.
The strong convex set-valued maximal operator is defined by
\begin{equation}\label{eq:BiParCSMO}
M_\calK^s F(x)
:= \overline{\conv}\Big(\bigcup_{R \in \calR} \langle F \rangle_R 1_R(x) \Big)
=\overline{\conv}(\{ \langle f \rangle_R 1_R(x) \colon R \in \calR, f \in S^0(F)\} ),
\end{equation}
where $F \colon \R^{n_1} \times \R^{n_2} \to \calK_{bcs}(\R^d)$ is locally integrably bounded.

In the bi-parameter context we will also use one-parameter convex set-valued maximal operators acting only in one of the parameters.
Let $F \colon \R^n \to \calK_{bcs}(\R^d)$ be locally integrably bounded. Define
$$
M_\calK^1F(x_1,x_2):= M_\calK(F(\cdot, x_2))(x_1).
$$
Similarly, we define the operator $M_\calK^2$ which acts on the second parameter.

\begin{rem}
Let $p \in (1, \infty)$ and $W \in \calA_p$ be a one-parameter matrix weight. There is a
variant of the Christ-Goldberg maximal function defined by
\begin{equation}\label{eq19}
M_W'f(x) = \sup_{Q \owns x}
\fint_Q | \mathcal{W}_{Q,p} W(y)^{-1} f(y)| \ud y.
\end{equation}
The auxiliary operator $M_W'$ was introduced in \cite{CG,Goldberg} and it was used to prove the matrix weighted estimates 
\eqref{eq16} of $M_W$. 
Here $\mathcal{W}_{Q,p}$ is a
so-called \emph{reducing operator} associated to $W$ on $Q$, see \cite[Section 6]{BC} or
\cite[Section 1]{Goldberg} (recall that \cite{BC} uses the convention \eqref{eq18} and \cite{Goldberg} the convention
\eqref{eq17}). The operator $\mathcal{W}_{Q,p}$ is a symmetric and positive definite $d\times d$ matrix such that
$$
\Big(\fint_Q |W(x) v|^p \ud x \Big)^{1/p}
\le |\mathcal{W}_{Q,p} v |
\le \sqrt{d} \Big(\fint_Q |W(x) v|^p \ud x \Big)^{1/p}
$$
for every $v \in \R^d$.

The bi-parameter analogue of $M_W'$, obtained by replacing cubes by rectangles in \eqref{eq19}, 
was shown to satisfy bi-parameter matrix weighted estimates in 
\cite[Proposition 4.1]{DKPS}. 
\end{rem}

\section{Main Theorems}\label{sec:MainThms}

As discussed in the introduction, the usual scalar strong maximal function $M^s$ 
satisfies the estimate $M^s f(x) \le M^1 M^2 f(x)$, 
where $f \colon \R^{n_1} \times \R^{n_2} \to \R$ and $M^i$ is the one-parameter maximal
function acting in the parameter $i$.
Since the convex set-valued maximal operators are suitable for iteration, we can prove
the following analogue of this.

\begin{lem}\label{lem:BiParDomination}
Let $F \colon \R^n \to \calK_{bcs}(\R^d)$ be locally integrably bounded
and assume also that $M_\calK^2 F$ is locally integrably bounded. Then
$$
M_\calK^s F \subset M_\calK^1M_\calK^2F.
$$

\begin{proof}
Fix some $x \in \R^{n_1} \times \R^{n_2}$ and some bi-parameter rectangle $R = Q_1 \times Q_2 \owns x$.
Let $f \in S^0(F)$. Then $f(y_1,y_2) \in F(y_1,y_2)$, and thus
$$
\fint_{Q_2} f(y_1,y_2) \ud y_2 
\in \fint_{Q_2} F(y_1,y_2) \ud y_2
\subset M_\calK^2 F(y_1,x_2).
$$
Integrating this over $Q_1$ gives
$$
\fint_R f \ud y
= \fint_{Q_1} \fint_{Q_2} f(y_1,y_2) \ud y_2 \ud y_1
\in  \fint_{Q_1} M_\calK^2 F(y_1,x_2) \ud y_1
\subset M_\calK^1 M_\calK^2 F(x).
$$
Since this holds for any selection function $f$ and any rectangle $R = Q_1 \times Q_2 \owns x $,
and since by definition $M_\calK^1M_\calK^2F(x)$ is closed and convex, 
it follows that 
$$
M_\calK^s F(x)
=\overline{\conv}(\{ \langle f \rangle_R 1_R(x) \colon R \in \calR, f \in S^0(F)\} ) 
\subset M_\calK^1M_\calK^2F(x).
$$
\end{proof}

\end{lem}

\begin{thm}\label{thm:StrongConvex}
Let $p \in (1, \infty]$ and $W \in \calA_p(\R^{n_1} \times \R^{n_2})$. Then
$$
\| M_\calK^s F \|_{L^p_{\calK}(W)}
\le C(n,d,p) [W]_{\calA_p(\R^{n_1} \times \R^{n_2})}^{2p'}  \| F \|_{L^p_{\calK}(W)}.
$$
If $p= \infty$, the estimate holds with $[W]_{\calA_\infty(\R^{n_1} \times \R^{n_2})}$ 
instead of $[W]_{\calA_\infty(\R^{n_1} \times \R^{n_2})}^{2}$.

\end{thm}

\begin{proof}
Suppose first that $p \in (1, \infty)$ and $F \in L^p_\calK(W)$. 
We begin by showing that $F$ and $M^2_\calK F$ are locally integrably bounded.

Let $Q \subset \R^n$ be a cube. Then
$$
\int_Q | F(y) | \ud y
\le \Big( \int_Q |W(y) F(y) |^p \ud y \Big)^{1/p}
\Big( \int_Q \OpN{W(y)^{-1}}^{p'} \ud y \Big)^{1/p'}.
$$
The first factor is finite and the second is less than
$$
\OpN{W(x)^{-1}} \Big( \int_Q \OpN{W(x)W(y)^{-1}}^{p'} \ud y \Big)^{1/p'}
$$
for almost every $x \in Q$. 
The matrix $\calA_p(\R^{n_1} \times \R^{n_2})$ condition implies that
the last integral is finite for almost every $x \in Q$. Thus $F$ is locally integrably bounded.
The proof below implies that $M^2_\calK F \in L^p_\calK(W)$,
and therefore we also get that $M^2_\calK F$ is locally integrably bounded.

Lemma \ref{lem:BiParDomination} implies that
$$
\int_{\R^n} | W(x) M_\calK F(x)|^p \ud x 
\le \int_{\R^n} | W(x) M_\calK^1M_\calK^2F(x)|^p \ud x. 
$$
Lemma \ref{lem:UniformAp}  says that 
$[W(\cdot, x_2) ]_{\calA_p(\R^{n_1})} \lesssim [ W ]_{\calA_p(\R^{n_1} \times \R^{n_2})}$   for almost every $x_2$.
Therefore, the matrix weighted estimate of $M_\calK$ \eqref{eq:ConvexMBounded}
implies that
\begin{equation*}
\begin{split}
\int_{\R^{n_2}} \int_{\R^{n_1}} &| W(x_1,x_2) M_\calK^1M_\calK^2F(x_1,x_2)|^p \ud x_1 \ud x_2 \\
&\lesssim [W]_{\calA_p(\R^{n_1} \times \R^{n_2})}^{pp'} \int_{\R^{n_2}} \int_{\R^{n_1}} | W(x_1,x_2) M_\calK^2F(x_1,x_2)|^p \ud x_1 \ud x_2.
\end{split}
\end{equation*}
We can change the order of integration and repeat the estimate, which concludes the proof in the case $p \in (1, \infty)$.

Suppose then that $p=\infty$ and $F \in L^\infty_\calK(W)$.
This case can be proved directly from the definition of $\calA_\infty(\R^{n_1} \times \R^{n_2})$,
similarly as the $L^\infty$-estimate of $M_W$ proved in  \cite[Proposition 6.8]{BC}.

Arguing similarly as in the case $p < \infty$ above implies that $F$ is locally integrably bounded.
Let $\calR_\Q$ denote the set of those bi-parameter rectangles
in $\R^{n_1} \times \R^{n_2}$ whose vertices have rational coordinates. Then,
$$
M_\calK^s F(x) 
= \overline{\conv}(\{ \langle f \rangle_R 1_R(x) \colon R \in \calR_\Q, f \in S^0(F)\} ),
$$
that is, in the definition of $M_\calK^s$ we can use the countable collection $\calR_\Q$, see \cite[Proposition 5.7]{BC}.
For every $R \in \calR_\Q$ there exists a set $N_R \subset R$ of measure zero so that
$$
\fint_R \OpN{W(x) W(y)^{-1}}\ud y \le [W]_{\calA_\infty(\R^{n_1} \times \R^{n_2})}, \quad \text{for every } x \in R \setminus N_R.
$$
The set $N:= \bigcup_{R \in \calR_\Q} N_R$ has measure zero and
$$
\fint_R \OpN{W(x) W(y)^{-1} }\ud y \le [W]_{\calA_\infty(\R^{n_1} \times \R^{n_2})}, 
\quad \text{for every } R \in \calR_\Q \text{ and } x \in R \setminus N.
$$

Suppose now that $x \in \R^n \setminus N$. Let $R \in \calR_\Q$ be such that $x \in R$ and let $f \in S^0(F)$.
Then
\begin{equation*}
\begin{split}
|W(x) \langle f \rangle_R|
= \Big| \fint_R W(x) W(y)^{-1}W(y) f(y) \ud y \Big| 
& \le  \fint_R \OpN{W(x) W(y)^{-1}} \ud y \|  W f  \| _{L^\infty (\R^n;\R^d)} \\
& \le [W]_{\calA_\infty(\R^{n_1} \times \R^{n_2})} \| F \|_{L^\infty_\calK(W)}.
\end{split}
\end{equation*}
Therefore, 
\begin{equation*}
\begin{split}
|W(x)M_\calK^s F(x)|
&=\big|W(x)\overline{\conv} (\{ \langle f \rangle_R 1_R(x) \colon R \in \calR_\Q, f \in S^0(F)\} )\big| \\
&=\big|\overline{\conv} (\{ W(x) \langle f \rangle_R 1_R(x) \colon R \in \calR_\Q, f \in S^0(F)\} )\big| \\
&= \big|\{ W(x) \langle f \rangle_R 1_R(x) \colon R \in \calR_\Q, f \in S^0(F)\} \big|
\le [W]_{\calA_\infty(\R^{n_1} \times \R^{n_2})} \| F \|_{L^\infty_\calK(W)}
\end{split}
\end{equation*} 
 for every $x \in \R^n \setminus N$, which
proves the claim.

\end{proof}

If $E \subset \R^n$ is a set with positive and finite measure and $f \colon E \to \R^d$ is integrable over $E$,
\emph{the convex body average of $f$ over $E$} is defined by
$$
\langle \! \langle f \rangle \! \rangle_E
:= \{ \langle \varphi f \rangle_E \colon \varphi \colon E \to [-1,1]\},
$$
see Nazarov, Petermichl, Treil, Volberg \cite{NPTV}. We record the following simple lemma
(see \cite[Section 2]{NPST}).

\begin{lem}\label{lem2}
Let $E$ be a set with positive and finite measure. Let $f \colon E \to \R^d$ be integrable. Then
$$
d^{-1} \langle | f | \rangle_E  
\le | \langle \! \langle f \rangle \! \rangle _E |
\le \langle | f | \rangle_E.
$$
\end{lem}

\begin{proof}
We include the proof for reader's convenience. It is clear that $| \langle \! \langle f \rangle \! \rangle _E |
\le \langle | f | \rangle_E$.

Let $f=(f_1, \dots, f_d)$, where the functions $f_i$ are the coordinate functions of $f$. Then
$$
\fint_E | f  | \ud x \le \sum_i \fint_E | f_i| \ud x
=\sum_i \fint_E \sign(f_i) f_i \ud x
\le \sum_i \Big| \fint_E \sign(f_i)f \ud x \Big|
\le d | \langle \! \langle f \rangle \! \rangle _E |.
$$
\end{proof}

The convex set-valued maximal operators can be written in terms of
the Christ-Goldberg maximal operators and selection functions
as indicated in the next lemma.

\begin{lem}\label{lem1}
Let $W \colon \R^n \to \calM_d$ be a matrix weight. Suppose that 
$F \colon \R^n \to \calK_{bcs}(\R^d)$ is a measurable function such that $W^{-1} F$ is locally 
integrably bounded. Then
\begin{equation}\label{eq3}
|W(x) M_\calK^s(W^{-1}F)(x)|
\sim_d \sup_{ f \in S^0(F)} M_W^s f(x).
\end{equation}

In particular, if $F(x) = \conv ( \{-f(x),f(x)\})$ for some $f \colon \R^n \to \R^d$, then
\begin{equation}\label{eq7}
|W(x) M_\calK^s (W^{-1}F)(x)| 
\sim_d M_W^s f(x).
\end{equation}

These estimates also hold with $M_\calK$ and $M_W$ in place of $M_\calK^s$ and $M_W^s$, respectively.

\end{lem}

\begin{proof}
Notice that
$$
S^0(W^{-1}F)
=\{ W^{-1} f \colon f \in S^0(F)\}.
$$
There holds that
\begin{equation}\label{eq2}
\begin{split}
|W(x) M_\calK^s(W^{-1}F)(x)|
&=\Big|W(x) \overline{\conv}\big(\{ \langle W^{-1}f \rangle_R 1_R(x) \colon R \in \calR, f \in S^0(F)\} \big) \Big| \\
&=\Big| \overline{\conv}\big(\{ W(x) \langle W^{-1}f \rangle_R 1_R(x) \colon R \in \calR, f \in S^0(F)\} \big) \Big| \\
&= \sup_{R,f} |W(x) \langle W^{-1}f \rangle_R| 1_R(x).
\end{split}
\end{equation}
It is clear that this last quantity is less than $\sup_{f \in S^0(F)} M_W^s f(x)$.

On the other hand, let $f \in S^0(F)$ and $x \in R \in \calR$. Lemma \ref{lem2} gives that 
\begin{equation*}
\begin{split}
\fint_R |W(x) W(y)^{-1} f(y) | \ud y 
&\le d \sup_{\varphi \colon R \to [-1,1]} \Big| \fint_R W(x) W(y)^{-1} \varphi(y) f(y)  \ud y \Big| \\
&\le d |W(x) M_\calK^s (W^{-1}F)(x)|;
\end{split}
\end{equation*}
the last step follows from \eqref{eq2} because $\varphi(y) f(y) \in F(y)$ since the values of $F$ are convex and symmetric.
Taking supremum over $R \owns x$ and $f \in S^0(F)$ gives that
$$
\sup_{f \in S^0(F)}M^s_W f(x) \le d |W(x) M_\calK^s (W^{-1}F)(x)|.
$$

To get the last claim, let $f \colon \R^n \to \R^d$ and define $F(x) = \conv(\{-f(x),f(x)\})$.
Then every selection function of $F$ is of the form $\varphi f$, where $\varphi \colon \R^n \to [-1,1]$.
This gives that 
$$
|W(x) M_\calK^s(W^{-1}F)(x)|
\sim_d \sup_{ \varphi \colon \R^n \to [-1,1]} M_W^s(\varphi f)(x)
=M_W^s f(x).
$$

Finally we notice that 
it makes no difference to the proof whether we have the one-parameter or the bi-parameter maximal functions.
\end{proof}

\begin{rem}
Lemma \ref{lem1} implies the unweighted boundedness of $M_\calK^s$ immediately.
Indeed, let $p \in (1, \infty]$ and $F \in L^p_\calK(\R^n)$. Let $M_I^s$
be the unweighted strong maximal operator ($M_W^s$ in the case that
$W(x)=I$, the identity matrix, for every $x$.) Then \eqref{eq2} gives that
$$
|M_\calK ^s F(x)| 
\le \sup_{ f \in S^0(F)} M_I^s f(x)
\le M^s (|F|)(x), 
$$
where in the last step we have the scalar strong maximal operator $M^s$ (see \eqref{M^s}) 
acting on the scalar-valued function $|F|$. Therefore,
$$
\| M^s_\calK F \|_{L^p_\calK(\R^n)}
\le \| M^s (|F|) \|_{L^p(\R^n)} 
\lesssim \||F| \|_{L^p(\R^n)}
=\| F \|_{L^p_\calK(\R^n)}.
$$

\end{rem}

The proof of the matrix weighted estimate of $M_\calK$ \eqref{eq:ConvexMBounded} in \cite[Theorem 6.9]{BC}
is based on the fact that 
Lemma \ref{lem1} holds in fact in the following stronger form.
Given $F \colon \R^n \to \calK_{bcs}(\R^d)$ and a matrix weight $W$ such that $W^{-1}F$ is locally integrably bounded,
there exist selection functions $f_1, \dots, f_d \in S^0(F)$ such that
\begin{equation}\label{eq4}
|W(x) M_\calK(W^{-1}F)(x)|
\lesssim_d \sum_{i=1}^d M_W f_i(x),
\end{equation}
see \cite[Proof of Theorem 6.9]{BC}. The analogous estimate holds with $M_\calK^s$ and $M_W^s$
in place of $M_\calK$ and $M_W$, respectively. Compare this with \eqref{eq3}. 
Equation \eqref{eq4} reduces the matrix weighted estimates of $M_\calK$ to the boundedness of $M_W$. 
In the proof of Theorem \ref{thm:StrongCG} we will use the connection between $M_\calK^s$ and 
$M_W^s$ in the other direction to deduce estimates for 
$M_W^s$ from the corresponding estimates of $M_\calK^s$.

\begin{proof}[Proof of Theorem \ref{thm:StrongCG}]
Let $f \in L^p(\R^n;\R^d)$. Define the function $F \in L^p_{\calK}(\R^n)$ by $F(x)=\conv(\{-f(x),f(x)\})$.
Then, by \eqref{eq7} there holds that $M^s_Wf \lesssim |W M_\calK^s(W^{-1}F)|$. Therefore,
\begin{equation}\label{eq8}
\begin{split}
\| M^s_W f \|_{L^p(\R^n)}
\lesssim \| W M_\calK^s (W^{-1}F) \|_{L_\calK^p(\R^n)}
&\lesssim [W]_{\calA_p(\R^{n_1} \times \R^{n_2})}^{2p'} \| F \|_{L^p_\calK(\R^n)} \\
&=[W]_{\calA_p(\R^{n_1} \times \R^{n_2})}^{2p'} \| f \|_{L^p(\R^n;\R^d)},
\end{split}
\end{equation}
where in the second step we used the matrix weighted estimate of $M_\calK^s$ (Theorem \ref{thm:StrongConvex}). 
If $p=\infty$ the argument gives the constant 
$[W]_{\calA_\infty(\R^{n_1} \times \R^{n_2})}$.
\end{proof}

\begin{rem}\label{rem:sharpness}
Let $p \in (1, \infty)$.
In Theorem \ref{thm:StrongCG} we got the constant $C(n,d,p)[W]_{\calA_p(\R^{n_1} \times \R^{n_2})}^{2p'}$
related to the boundedness of $M_W^s$. It is not known if the exponent $2p'$,
which we obtained by twice iterating the one-parameter maximal function estimate, 
is best possible. 
More precisely, this is not known in the scalar case $d=1$.

On the other hand, it is well known that in the scalar case $d=1$ the lower bound 
$\| M^s_W \|_{L^p \to L^p} \ge [W]_{\calA_p(\R^{n_1} \times \R^{n_2})}^{p'}$
can be easily proved by testing the maximal operator with functions of the form
$1_R W^{-p'/p}$. This lower bound holds also in the case $p=\infty$,
so the constant obtained in Theorem \ref{thm:StrongCG} in the case $p=\infty$ 
is sharp.

We remind the reader that we use the formulation \eqref{eq18} of weighted estimates and the related weight class 
$\calA_p$, and not the traditional way \eqref{eq17} and the related weight class $A_p$ 
(see the end of Section \ref{sec:MatrixWeights} for $A_p$). This is to be kept in mind when one compares
constants in different papers.
\end{rem}

Finally, corresponding to the estimate \eqref{eq9} of $M_W$, 
we record the following version of the matrix weighted estimates
of $M_\calK^s$ and $M_W^s$.

\begin{thm}\label{thm1}
Let $p \in (1, \infty)$ and $B > 1$.
There exists a
$
\delta 
= \delta(n,d,p,B) > 0
$ 
such that if $W \in \calA_p(\R^{n_1} \times \R^{n_2})$ with 
$[W]_{\calA_p(\R^{n_1} \times \R^{n_2})} \le B$,
the estimates
\begin{equation}\label{eq10}
\|M_W^s f \|_{L^q(\R^n)} \le C(n,d,p,q, B) \| f \|_{L^q(\R^n;\R^d)}
\end{equation}
and
\begin{equation}\label{eq11}
\|M_\calK^s F \|_{L^q_\calK(W)} \le C(n,d,p,q, B) \| F \|_{L^q_\calK(W)}
\end{equation}
hold for those $q \in (1, \infty)$ such that $|q-p| < \delta$.
\end{thm}

\begin{proof}
We begin by proving the one-parameter version of \eqref{eq11}.
Equation \eqref{eq9} says that there exists a 
$\delta = \delta (n,d,p, B) > 0$ such that
if $W \in \calA_p(\R^n)$ is a one-parameter matrix weight with $[W]_{\calA_p(\R^n)} \le B$,
then
$$
\| M_W f \|_{L^q(\R^n)}
\le C(n,d,q,B) \| f \|_{L^q(\R^n;\R^d)}
$$
holds for those $q \in (1, \infty)$ such that $|q-p| < \delta$.
For such a $q$ let $G \in L^q_\calK(\R^n)$. Equation \eqref{eq4}
implies that there exist selection functions $g_1, \dots, g_d \in S^0(G)$ so that
$$
|W(x) M_\calK G(x)|
= |W(x) M_\calK(W^{-1} W G)(x)|
\lesssim \sum_i M_W (Wg_i)(x).
$$
This leads to
$$
\| M_\calK G \|_{L^q_\calK(W)}
\lesssim \sum_i \| M_W (Wg_i) \|_{L^q(\R^n)}
\lesssim \sum_i \|  Wg_i \|_{L^q(\R^n;\R^d)}
\lesssim \| G \|_{L^q_\calK(W)}.
$$

Now we turn to \eqref{eq11}. Let $W \in \calA_p(\R^{n_1} \times \R^{n_2})$
be such that $[W]_{\calA_p(\R^{n_1} \times \R^{n_2})} \le B$.
By Lemma \ref{lem:UniformAp} 
$$
[W(x_1, \cdot)]_{\calA_p(\R^{n_2})} \le C(n,d,p,)[W]_{\calA_p(\R^{n_1} \times \R^{n_2})}
\le C(n,d,p,) B
$$
for almost every $x_1 \in \R^{n_1}$. Therefore, applying the one-parameter version of \eqref{eq11} 
(with $B' = C(n,d,p,)B$)
it follows that 
there exists a 
$
\delta = \delta(n,d,p,B) > 0
$
such that for a.e.\ $x_1$ the estimate 
$$
\| M_\calK G \|_{L^q_\calK(W(x_1, \cdot))}
\le C(n,d,p,q, B) 
\|  G \|_{L^q_\calK(W(x_1, \cdot))}
$$
holds for those $q \in (1, \infty)$ such that $|q-p| < \delta$.
The corresponding estimate holds with the roles of $x_1$ and $x_2$ reversed.

Now, using the fact that $M^s_\calK F \subset M^1_\calK M^2_\calK F$ (Lemma \ref{lem:BiParDomination}) 
we conclude that
\begin{equation*}
\begin{split}
\int_{\R^n} | W(x) M_\calK^s F (x) |^q \ud x
&\le \int_{\R^{n_2}} \int_{\R^{n_1}} | W(x_1,x_2) M^1_\calK M^2_\calK F (x_1,x_2) |^q \ud x_1 \ud x_2 \\
& \lesssim  \int_{\R^{n_1}} \int_{\R^{n_2}} | W(x_1,x_2)  M^2_\calK F (x_1,x_2) |^q \ud x_2 \ud x_1 \\
& \lesssim \int_{\R^n} | W(x) F (x) |^q \ud x
\end{split}
\end{equation*} 
holds for those $q \in (1, \infty)$ such that $|q-p| < \delta$.

The estimate \eqref{eq10} follows from \eqref{eq11}
via \eqref{eq7}, similarly as in \eqref{eq8}.
\end{proof}

\subsection{Extrapolation with bi-parameter matrix weights}\label{sec:bi-parExtra}
In \cite[Theorem 9.1]{BC} the Rubio de Francia extrapolation with matrix weights was proved.
One of the main ingredients was the convex set-valued maximal operator $M_\calK$ and its
matrix weighted estimates. 

Theorem \ref{thm:StrongConvex} gives the bi-parameter matrix weighted 
estimates of $M_\calK^s$.
Using these one
proves the extrapolation with bi-parameter matrix weights precisely following the proof  in \cite{BC}. 
To do this one replaces cubes by rectangles and the one-parameter maximal function $M_\calK$ by the 
bi-parameter maximal function $M_\calK^s$.
After this the proof goes through without any significant changes in the details. 
In general, the passage from cubes to rectangles, from one-parameter analysis
to bi-parameter analysis, makes a big difference. 
However, concerning the results needed for extrapolation, there
is no essential difference once one has the matrix weighted estimates of $M_\calK^s$
(see \cite{CMP_Book} for this fact in the scalar case). 

The operator $M_\calK^s$ has the same essential properties
as $M_\calK$: 
it is monotone, sublinear and bounded on the matrix weighted $L^p$ spaces.
The only difference between $M_\calK^s$ and $M_\calK$ is in
how the constants in the matrix weighted estimates depend
on the weight characteristics. Therefore, when one follows
the proof of \cite{BC} in the bi-parameter context, one ends up with 
different quantitative constants.

The statement of the extrapolation theorem involves  \emph{extrapolation pairs}, see \cite[Section 9]{BC}
or the book Cruz-Uribe, Martell, P\'erez \cite{CMP_Book}. 
Let $\calF$ denote a family of pairs $(f,g)$ of measurable functions $f,g \colon \R^n \to \R^d$.
If we write an inequality 
$$
\| f \|_{L^p(W)} \le C \| g \|_{L^p(W)}, \quad (f,g) \in \calF,
$$
we mean that this inequality holds for all pairs $(f,g) \in \calF$ such that 
$\| f \|_{L^p(W)} < \infty$. Here 
$\| f \|_{L^p(W)}:= \| | Wf| \|_{L^p(\R^n)}$.

If $q \in [1, \infty]$, then $q'$ denotes the dual exponent of $q$, that is, 
$1/q+1/q' =1$.

\begin{thm}\label{thm:BiParExtrap}
Let $p_0 \in [1, \infty]$. Suppose that there exists an increasing function 
$K_{p_0}$ such that
$$
\| f \|_{L^{p_0}(W)} \le K_{p_0}([W]_{\calA_{p_0}(\R^{n_1} \times \R^{n_2})})\| g \|_{L^{p_0}(W)},
 \quad (f,g) \in \calF,
$$
holds for every $W \in \calA_{p_0}(\R^{n_1} \times \R^{n_2})$.
Then, for every $p \in (1,\infty)$ the estimate
\begin{equation}\label{eq5}
\| f \|_{L^{p}(W)} \le K(K_{p_0},p,p_0,n,d,[W]_{\calA_{p}(\R^{n_1} \times \R^{n_2})})\| g \|_{L^{p}(W)},
 \quad (f,g) \in \calF,
\end{equation}
holds for every $W \in \calA_{p}(\R^{n_1} \times \R^{n_2})$, where
$K(K_{p_0},p,p_0,n,d,[W]_{\calA_{p}(\R^{n_1} \times \R^{n_2})})$ is increasing with
respect to $[W]_{\calA_p(\R^{n_1} \times \R^{n_2})}$.

\end{thm}

\begin{rem}
The one-parameter extrapolation theorem with matrix weights proved in \cite[Theorem 9.1]{BC}
is a \emph{sharp constant extrapolation theorem}. 
This means that  in \cite[Theorem 9.1]{BC}
the resulting constant $K(K_{p_0},p,p_0,n,d,[W]_{\calA_{p}(\R^{n})})$,
which corresponds to the constant in \eqref{eq5}, is given
by a certain formula and depends on $[W]_{\calA_p(\R^n)}$
in a certain sharp way. The bi-parameter theorem \ref{thm:BiParExtrap}
is proved by following the proof in \cite{BC}, but one does not get the same quantitative
constant. This is because the constant in the matrix weighted estimate of $M_\calK^s$
we get in Theorem \ref{thm:StrongConvex} involves $[W]^{2p'}_{\calA_p(\R^{n_1} \times \R^{n_2})}$,
while in the corresponding one-parameter estimate there is $[W]_{\calA_p(\R^n)}^{p'}$, see \eqref{eq:ConvexMBounded}.
One could track the constants in \cite{BC} to obtain some quantitative constant also in 
Theorem \ref{thm:BiParExtrap}.
\end{rem}

Bi-parameter Journ\'e operators, also known as bi-parameter Calder\'on-Zygmund operators,
are a certain type of bi-parameter singular integral operator. We refer for example to \cite{DKPS} and  
Martikainen \cite{Martikainen} for the exact definitions. It was shown in \cite{DKPS} that
if $T$ is a bi-parameter Journ\'e operator and $W \in \calA_2(\R^{n_1} \times \R^{n_2})$
is a bi-parameter matrix weight, then the estimate
\begin{equation}\label{eq12}
\| T f \|_{L^2(W)}
\lesssim \| f \|_{L^2(W)}
\end{equation}
holds. The authors in \cite{DKPS} made significant progress towards extending
\eqref{eq12} to the whole range $p \in (1, \infty)$. The final missing part was a
Fefferman-Stein type estimate for the one-parameter 
Christ-Goldberg maximal function $M_W$, see \cite[Question 7.3, Remark 8.14]{DKPS}.
They also showed that if in addition $T$ is \emph{paraproduct free}, then
$$
\| T f \|_{L^p(W)}
\lesssim \| f \|_{L^p(W)}
$$
holds for $p \in (1, \infty)$ and $W \in \calA_p(\R^{n_1} \times \R^{n_2})$.

In \cite{DKPS} the matrix weights are allowed to have
complex entries. In this paper and in \cite{BC} the matrix weights have real entries.
See the discussion of this question at the end of the Introduction of \cite{DKPS}.
Two-weight estimates with bi-parameter matrix weights 
are also considered in \cite{DKPS};
see \cite[Theorem 1.1]{DKPS}.

The bi-parameter extrapolation, Theorem \ref{thm:BiParExtrap},
extends the estimate \eqref{eq12} to the whole range $p \in (1, \infty)$.

\begin{thm}\label{thm:BiParJourne}
Suppose that $T$ is a bi-parameter Journ\'e operator. Let
$p \in (1, \infty)$ and $W \in \calA_p(\R^{n_1} \times \R^{n_2})$ be a bi-parameter matrix weight.
Then the estimate
$$
\| T f \|_{L^p(W)}
\lesssim \| f \|_{L^p(W)}
$$
holds.
\end{thm}

\bibliography{MatrixWeights}
\bibliographystyle{abbrv}

\end{document}